\documentclass[12pt]{amsart}  

\usepackage{custom_preamble_arXiv} 

\begin{document}

\title{A statistical approach to covering lemmas}

\author{\tsname}
\address{\tsaddress}
\email{\tsemail}

\begin{abstract}
We discuss a statistical variant of Ruzsa's covering lemma and use it to show that if $G$ is an Abelian group of bounded exponent and $A \subset G$ has $|A+A| \leq K|A|$ then the subgroup generated by $A$ has size at most $\exp(O(K\log^22K))|A|$, where the constant in the big-$O$ depends on the exponent of the group only.
\end{abstract}

\maketitle

\section{Introduction}

In this note we are concerned with the paper \cite{ruz::01} of Ruzsa in which he proved what is now called the Ruzsa covering lemma, and where he went on to give its prototypical application to the Fre{\u\i}man-Ruzsa theorem.  In this note we shall examine a statistical variant of Ruzsa's covering lemma and show how to use it to improve that same application.
\begin{theorem}[Fre{\u\i}man-Ruzsa theorem for Abelian groups of bounded exponent]\label{thm.fr}
Suppose that $G$ is an Abelian group of exponent\footnote{Recall that an Abelian group has \textbf{exponent} $r$ if $r$ is the minimal natural number such that $rx=0_G$ for every $x \in G$.} $r$ and $\emptyset \neq A \subset G$ has $|A+A| \leq K|A|$.  Then there is a function $F$ depending only on $r$ and $K$ such that $\langle A \rangle$, the group generated by $A$, has size at most $F(r,K)|A|$.
\end{theorem}
We shall take $F(r,K)$ to be the point-wise smallest function such that the conclusion of this theorem holds.

Ruzsa proved that $F(r,K) \leq r^{O(K^4)}$ (in \cite{ruz::01}) and noted by considering sets of independent elements that $F(r,K) \geq r^{\Omega(K)}$; he further conjectured that this was close to optimal in particular suggesting that $F(r,K) \leq r^{O(K)}$.

We shall return to Ruzsa's conjecture and the cases in which it is known shortly, but first we shall take a brief look at the proof of Theorem \ref{thm.fr} from \cite{ruz::01}.  The argument has two parts: first, note that if $X\subset G$ has
\begin{equation}\label{eqn.covering}
A+A \subset X+A,
\end{equation}
then by induction $\langle A \rangle \subset \langle X \rangle + A$ and so
\begin{equation*}
|\langle A\rangle| \leq |\langle X \rangle||A| \leq r^{|X|}|A|.
\end{equation*}
All we need to do now is find a set $X$ satisfying (\ref{eqn.covering}) that is as small as possible -- this will be the second part of the argument.

Unfortunately, if $G$ is finite (but large) and $A$ is chosen uniformly at random from subsets of $G$ of size $|G|/K$ then $|A+A| \leq K|A|$ and with high probability all $X$s satisfying (\ref{eqn.covering}) have $|X|=\Omega(K\log |G|)$.

One can eliminate the problem presented by random sets by considering not $A$, but instead $A-A$ \emph{i.e.} finding a set $X$ such that
\begin{equation}\label{eqn.covering2}
(A-A)+(A-A) \subset X+(A-A).
\end{equation}
It is now possible to find a set $X$ whose size depends only on $K$ -- this is what we now call Ruzsa's Covering Lemma (see \cite[Lemma 2.14]{taovu::}).  Dealing with $A-A$ comes at a price because the set $X$ in (\ref{eqn.covering2}) may need to be as large as $K^3$; this is what has (do date) prevented us from establishing Ruzsa's conjecture by this route.

Instead of working with $A-A$ we shall continue to work with $A$, but (necessarily) relax (\ref{eqn.covering}) to a statistical statement of `almost covering'\footnote{The reader interested primarily in this may skip directly to \S\ref{sec.statcov}.}.  This too comes at a price, because the inductive consequence of (\ref{eqn.covering}) need no longer hold.  Nevertheless, for considerably more work, we are able to handle this and prove the following version of Theorem \ref{thm.fr}.
\begin{theorem}[Fre{\u\i}man-Ruzsa theorem for Abelian groups of bounded exponent]\label{thm.main}
Suppose that $G$ is an Abelian group of exponent $r$ and $\emptyset \neq A \subset G$ has $|A+A| \leq K|A|$.  Then the group generated by $A$ has size at most $\exp(O_r(K\log^{2}2K))|A|$.
\end{theorem}
Finally, before proceeding to the proof we should like to make a few remarks about this result.
\begin{enumerate}
\item The dependence on $r$ is poor with the proof giving
\begin{equation*}
F(r,K) \leq  \exp(O(K(\log 2K)((\log 2K)(\log r) + r^2))).
\end{equation*}
This is far from Ruzsa's conjecture, and means that unless $r$ is smaller than about $\log 2K$ there are better results available \emph{e.g.} those of Konyagin in (\ref{pt.6}).
\item\label{pt.6} The improvement on what was already known is very minor in two respects.  First, Schoen showed in \cite{sch::1} that
\begin{equation*}
F(r,K) \leq \exp(O_r(K^{1+o(1)})),
\end{equation*}
and arguments of Konyagin \cite{kon::1} can be used to show that
\begin{equation*}
F(r,K) \leq \exp(O_r(K\log^{3+o(1)}2K)),
\end{equation*}
so we are only saving a power of $\log 2K$, and less than that unless $r$ is constant.

Secondly, in the case when the exponent is a prime \emph{much} more precise estimates are known.  Indeed, $F(2,K)$ has been completely determined by Zohar \cite{zoh::} using the compression techniques introduced by Green and Tao \cite{gretao::2}, and Zohar's argument was extended by Lovett and Zohar in \cite{lovzoh::} to show that
\begin{equation*}
F(r,K) \leq \frac{{r}^{2K-1}}{2K-1}
\end{equation*}
when $K \geq 8$ and $r$ is a prime, which is essentially tight.
\item Although our results are slight, the conjecture itself is nowhere near as significant as the second conjecture mention in \cite{ruz::01} -- Marton's conjecture, also called the Polynomial Fre{\u\i}man-Ruzsa conjecture.  This has many applications and the arguments of both Konyagin and Schoen mentioned earlier both bear on this much more important conjecture.
\end{enumerate}

\section{A statistical covering lemma}\label{sec.statcov}

Our starting point, then, is Ruzsa's covering lemma which is proved in \cite{ruz::01} although is not explicitly separated out there and appears as a distinct result in \cite[Lemma 2.14]{taovu::}.
\begin{lemma}[Ruzsa's covering lemma]
Suppose that $|A+B| \leq K|B|$.  Then there is a set $X \subset B$ of size at most $K$ such that
\begin{equation*}
A \subset X+(B-B).
\end{equation*}
\end{lemma}
The classic proof of this is to let $X$ be a maximal $B$-separated subset of $A$, meaning let $X \subset A$ be maximal such that $(x+B) \cap (x'+B) = \emptyset$ for all $x\neq x' \in X$.  A short argument then yields the lemma.

The statistical covering lemma we shall use follows below using essentially the same argument; to the extent that it is different the additional ideas can be seen in the Green-Ruzsa covering lemma \cite[Lemma 2.1]{greruz::}.
\begin{lemma}[Statistical covering lemma]\label{lem.statcov}  Suppose that $|A+B| \leq K|B|$ and $\delta \in (0,1]$ is a parameter.  Then there is a set $X \subset A$ of size at most $\delta^{-1}(K-1)+1$ such that
\begin{equation*}
|(x+B) \cap (X+B)| \geq (1-\delta)|B| \text{ for all } x \in A.
\end{equation*}
\end{lemma}
\begin{proof}
We construct sets $X_0,\dots,X_k$ iteratively; let $x_0 \in A$ and $X_0:=\{x_0\}$.  Suppose we have constructed $X_i\subset A$.  If there is some $x_{i+1}\in A$ such that
\begin{equation*}
|(x_{i+1}+B) \cap (X_i + B)| < (1-\delta)|B|
\end{equation*}
then let $X_{i+1}:=X_i \cup \{x_{i+1}\}$; if there is no such $x_{i+1}$ then terminate with $X:=X_i$.  It follows by induction that $|X_i| \leq i+1$.  Moreover, if there is a suitable $x_{i+1}$ then
\begin{eqnarray*}
|(x_{i+1}+B) \setminus (X_i + B)| & \geq & |x_{i+1}+B| - |(x_{i+1}+B) \cap (X_i + B)|\\ & > & |B| - (1-\delta)|B| = \delta |B|.
\end{eqnarray*}
It follows that
\begin{eqnarray*}
|X_i+B| &\geq & |(X_i+B) \setminus (X_{i-1}+B)| + |(X_{i-1}+B) \setminus (X_{i-2}+B)|\\ & & +\dots + |(X_1+B)\setminus (X_0+B)| + |X_0+B|\\ &\geq &|(x_i+B) \setminus (X_{i-1}+B)| + |(x_{i-1}+B) \setminus (X_{i-2}+B)|\\ & & +\dots + |(x_1+B)\setminus (X_0+B)| + |x_0+B|\\ & > & \delta |B|.i +|B|.
\end{eqnarray*}
On the other hand $X_i \subset A$ and so $|X_i+B| \leq K|B|$, and hence the iteration terminates for some $i < \delta^{-1}(K-1)$, and which point we have
\begin{equation*}
|(x+B) \cap (X_i + B)| \geq(1-\delta)|B| \text{ for all } x \in A.
\end{equation*}
The bound follows since $X=X_i$ and $|X_i| \leq i+1 < \delta^{-1}(K-1)+1$ as required.
\end{proof}
The case $\delta=1$ is essentially the argument for Ruzsa's covering lemma since
\begin{equation*}
|(x+B)\cap (X+B)|>0 \text{ if and only if } x \in X+B-B.
\end{equation*}
In this note, however, we are interested in the case of small $\delta$ and $A=B$, and for convenience we shall wrap up the conclusion of the lemma in the case $A=B$ in a definition.  We say that $A$ is \textbf{$(1-\delta)$-covered by $X$} if
\begin{equation*}
|(x+A) \cap (X+A)| \geq(1-\delta) |A| \text{ for all }x \in A.
\end{equation*}
Thus $A$ is $1$-covered by $X$ if and only if $A+A\subset X+A$, and the Statistical covering lemma tells us that if $|A+A| \leq K|A|$ then $A$ is $(1-\delta)$-covered by a set $X$ of size $O(\delta^{-1}K)$.

To use this definition it will be useful to relate it to convolution.  Define the translation operator on $\ell_p(G)$ for $p\in [1,\infty]$ in the usual way \emph{viz}.
\begin{equation*}
\tau_x:\ell_p(G) \rightarrow \ell_p(G); f \mapsto (y \mapsto f(x+y)).
\end{equation*}
Given two functions $f,g \in \ell_2(G)$ we then define their \textbf{convolution} point-wise by
\begin{equation*}
f\ast g(x):=\langle \tau_{-x}(f),g\rangle_{\ell_2(G)} = \sum_{y+z=x}{f(y)g(z)} \text{ for all }x \in G.
\end{equation*}

Now, if $A$ is $(1-\delta)$-covered by $X$ then it is possible to show (when $\delta=0$ this is essentially the induction we discussed in the introduction) that
\begin{equation}\label{eqn.apxcov}
\langle \overbrace{1_A \ast \dots \ast 1_A}^{(k+1) \text{ times}}, 1_{kX+A}\rangle_{\ell_2(G)} \geq (1-\delta)^k|A|^{k+1}.
\end{equation}
This already captures a great deal about $A$, in particular that $A$ has a lot of \emph{very} large Fourier coefficients.

Inequalities of the form (\ref{eqn.apxcov}) are not quite enough for us though, and to see what we need re-write the left hand side as a sum over $A^k$:
\begin{equation*}
\langle \overbrace{1_A \ast \dots \ast 1_A}^{(k+1) \text{ times}}, 1_{kX+A}\rangle_{\ell_2(G)} =\sum_{x \in A^k}{\langle\tau_{-(x_1+\dots + x_k)}(1_{A}),1_{kX+A}\rangle_{\ell_2(G)}}.
\end{equation*}
We shall want lower bounds on sums over sets that are slightly more general than the product set $A^k$; the definition of these is our next task.

\subsection{Generalised product sets} We think of $A^k$ as a product of uniform probability spaces, so that it is itself endowed with the uniform probability measure.  A set $\mathcal{A}\subset A^k$ is said to \textbf{contain a $\nu$-large generalised sub-product of $A^k$} for some $\nu \in (0,1]^{k}$ if there are sets $(\mathcal{A}_i)_{i=0}^k$ such that
\begin{enumerate}
\item\label{ax.1} (Start and end) $\mathcal{A}_0=\{()\}$ and $\mathcal{A}_k \subset \mathcal{A}$;
\item\label{ax.2} (Powers) whenever $1\leq i \leq k$ we have $\mathcal{A}_i \subset A^i$;
\item\label{ax.3} (Sub-martingale) whenever $1 \leq i \leq k$ we have the point-wise inequality
\begin{equation*}
\E_{a_{i}}{1_{\mathcal{A}_{i}}(a_1,\dots,a_{i})} \geq \nu_{i} 1_{\mathcal{A}_{i-1}}(a_1,\dots,a_{i-1}).
\end{equation*}
\end{enumerate}
In many cases we shall have $\nu$ a constant vector: if $\epsilon \in (0,1]$ then we say that $\mathcal{A}$ contains an $\epsilon$-large generalised sub-product of $A^k$ if it contains a $\nu$-large generalised sub-product of $A^k$ for $\nu=\nu(\epsilon) \in (0,1]^k$ defined by $\nu_i=\epsilon$ for $1 \leq i \leq k$.

To explain the heading for \ref{ax.3} write $X_i:=1_{\mathcal{A}_i \times A^{k-i}}$ considered as random variables on $A^k$, and $Y_k,\dots,Y_1$ for the coordinate projection maps $A^k\rightarrow A; a \mapsto a_i$ considered as random elements.  \ref{ax.3} is simply the statement that $((\prod_{j=1}^i{\nu_j})X_i)_{i=0}^k$ a (finite) sub-martingale with respect to $(Y_i)_{i=1}^k$.  That being said we shall make no use of any theory of martingales, and the reader not familiar with this may safely ignore this remark.
\begin{example}[Sub-products]\label{ex.prod}
If $A_1,\dots,A_k \subset A$ have size at least $\nu_1|A|,\dots, \nu_k|A|$ respectively then $\mathcal{A}:=A_1\times \dots \times A_k$ contains a $\nu$-large generalised sub-product of $A^k$.  To see this define
\begin{equation*}
\mathcal{A}_i:=A_1\times \dots \times A_i \text{ for } 0 \leq i \leq k.
\end{equation*}
We certainly have \ref{ax.1} (with the convention that the empty product is just the set containing the empty tuple) and \ref{ax.2}; it remains to check that
\begin{equation*}
\E_{a_{i}}{1_{\mathcal{A}_{i}}(a_1,\dots,a_{i})}  = \E_{a_{i}}{1_{\mathcal{A}_{i-1}}(a_1,\dots,a_{i-1})1_{A_{i}}(a_{i})} \geq \nu_{i} 1_{\mathcal{A}_{i-1}}(a_1,\dots,a_{i-1})
\end{equation*}
whenever $1 \leq i \leq k$ which gives \ref{ax.3} as required.
\end{example}
There are two properties of sub-products which are useful to us and which extend to generalised sub-products.  The first is that it is easy to compute the size of sub-products; we do this for generalised sub-products in Lemma \ref{lem.basicprop}.  The second is that we can apply the inclusion-exclusion principle fibre-wise.  To see why this is useful consider the example when $A_1,\dots,A_k,A_1',\dots,A_k' \subset A$ all have size at least $(1-\eta)|A|$.  Then
\begin{equation*}
|(A_1\times \dots \times A_k) \cap (A_1' \times \dots \times A_k')|\geq(1-2\eta)^k|A|^k
\end{equation*}
by the inclusion-exclusion principle applied in each fibre.  But if $1/k \lll \eta \lll 1$ this is much better than the bound we get from applying the inclusion-exclusion principle directly, ignoring the product structure, in which case we get
\begin{equation*}
|(A_1\times \dots \times A_k) \cap (A_1' \times \dots \times A_k')|\geq (2(1-\eta)^k-1)|A|^k.
\end{equation*}
We address this issue in Lemma \ref{lem.heredint}.
\begin{lemma}[Generalised sub-products, size bound]\label{lem.basicprop}
Suppose that $\mathcal{A}$ contains a $\nu$-large generalised sub-product of $A^k$. Then $|\mathcal{A}| \geq (\prod_{i=1}^k{\nu_i})|A|^k$.
\end{lemma}
\begin{proof}
Start by noting that for $1 \leq i \leq k$ we have, by \ref{ax.3} and linearity of expectation that
\begin{equation*}
\E_{a_1,\dots,a_{i}}{1_{\mathcal{A}_i}(a_1,\dots,a_{i})} \geq \nu_{i} \E_{a_1,\dots,a_{i-1}}{1_{\mathcal{A}_{i-1}}(a_1,\dots,a_{i-1})}.
\end{equation*}
But then by induction and \ref{ax.1} we have
\begin{equation*}
\E{1_{\mathcal{A}}} \geq \E{1_{\mathcal{A}_k}} \geq \prod_{i=1}^k{\nu_i}\E{1_{\mathcal{A}_0}} =  \prod_{i=1}^k{\nu_i},
\end{equation*}
and the result is proved.
\end{proof}
\begin{lemma}[Generalised sub-products, intersections]\label{lem.heredint}
Suppose that $\eta,\eta' \in [0,1)^k$ are such that $\eta+\eta' \in [0,1)^k$ and $\mathcal{A}$ and $\mathcal{A}'$ contain $(1-\eta)$-large and $(1-\eta')$-large generalised sub-products of $A^k$ respectively.  Then $\mathcal{A}\cap \mathcal{A}'$ contains a $(1-(\eta+\eta'))$-large generalised sub-product of $A^k$.
\end{lemma}
\begin{proof}
We start by defining $(\mathcal{A} \cap \mathcal{A}')_i:=\mathcal{A}_i \cap \mathcal{A}'_i$ for $0 \leq i \leq k$ so that \ref{ax.1} and \ref{ax.2} are satisfied for $\mathcal{A}\cap \mathcal{A}'$.  With explanations of the passage from line to line in the following paragraph, it remains to note that
\begin{eqnarray*}
\E_{a_{i}}{1_{(\mathcal{A} \cap \mathcal{A}')_{i}}(a_1,\dots,a_{i})} & = & \E_{a_{i}}{1_{\mathcal{A}_{i}}(a_1,\dots,a_{i})1_{\mathcal{A}'_{i}}(a_1,\dots,a_{i})}\\ & \geq  &  1_{(\mathcal{A} \cap \mathcal{A}')_{i-1}}(a_1,\dots,a_{i-1}) \E_{a_{i}}{1_{\mathcal{A}_{i}}(a_1,\dots,a_{i})1_{\mathcal{A}'_{i}}(a_1,\dots,a_{i})}\\ & \geq & 1_{(\mathcal{A} \cap \mathcal{A}')_{i-1}}(a_1,\dots,a_{i-1})\\ & & \times \E_{a_{i}}{(1_{\mathcal{A}_{i}}(a_1,\dots,a_{i})+1_{\mathcal{A}'_{i}}(a_1,\dots,a_{i})-1)}\\ & \geq & 1_{(\mathcal{A} \cap \mathcal{A}')_{i-1}}(a_1,\dots,a_{i-1})\\ & & \times ((1-\eta_i)1_{\mathcal{A}_{i-1}}(a_1,\dots,a_{i-1}) + (1-\eta'_i)1_{\mathcal{A}'_{i-1}}(a_1,\dots,a_{i-1}) -1)\\ & = & 1_{(\mathcal{A} \cap \mathcal{A}')_{i-1}}(a_1,\dots,a_{i-1})((1-\eta_i) + (1-\eta'_i) -1)\\ & = &1_{(\mathcal{A} \cap \mathcal{A}')_{i-1}}(a_1,\dots,a_{i-1})(1-(\eta_i +\eta'_i))
\end{eqnarray*}
for $1 \leq i \leq k$.  The passage from the first to the second line is because $1_{(\mathcal{A}\cap \mathcal{A}')_{i-1}}$ is Boolean and the two expectations are non-negative; from the second to the third is that $xy \geq x+y-1$ whenever $x,y \in \{0,1\}$; from the third to the fourth uses linearity of expectation and \ref{ax.3} for $\mathcal{A}$ and $\mathcal{A}'$; finally the fourth to the fifth uses the fact that $1_{(\mathcal{A}\cap \mathcal{A}')_{i-1}}1_{\mathcal{A}_{i-1}}=1_{(\mathcal{A}\cap \mathcal{A}')_{i-1}}$ by definition and similarly for $\mathcal{A}'$.  The result is proved.
\end{proof}
The next lemma captures how the notion of $(1-\delta)$-covering (from the start of this section) is related to generalised sub-products.  If $A$ is $1$-covered by a set $X$ containing $0_G$ then a short induction tells us that for any $x \in A$ we have
\begin{equation*}
\{a \in A^k: x+\sum_{i}{a_i} \in kX+A\} = A^k;
\end{equation*}
the set has a product structure.  If it is almost $1$-covered then it has a generalised sub-product structure.  Although long-winded the proof below is straight-forward.
\begin{lemma}\label{lem.heredwant}
Suppose that $A$ is $(1-\delta)$-covered by $X\ni 0_G$ for some $\delta \in [0,1)$.  Then for any $x \in A$, $k \in \N$ and $S \subset [k]$, the set
\begin{equation*}
\mathcal{A}^S:=\{a \in A^k: x+\sum_{s \in S}{a_s} \in \overbrace{X+\dots +X}^{|S| \text{ times}}+A\}
\end{equation*}
contains a $\nu(S)$-large generalised sub-product where $\nu(S)_i=1-\delta1_S(i)$.
\end{lemma}
\begin{proof}
For each $S$ we shall construct sets $(\mathcal{A}^S_i)_{i=0}^k$ satisfying \ref{ax.1} (for $\mathcal{A}^S$), \ref{ax.2}, and \ref{ax.3} with $\nu(S)$.  When $S=\emptyset$ we have
\begin{equation*}
\{a \in A^k: x+\sum_{s \in S}{a_s} \in \overbrace{X+\dots +X}^{|S| \text{ times}}+A\} = \{a \in A^k: x \in A\}=A^k
\end{equation*}
since $0_G \in X$ and $x \in A$, and so we have suitable sets $(\mathcal{A}^S_i)_{i=0}^k$ from Example \ref{ex.prod}, namely $\mathcal{A}^S_i=A^i$.

Now, suppose that $S$ is non-empty, that the largest element of $S$ is $j$, and that we have constructed suitable sets $(\mathcal{A}^{S'}_i)_{i=0}^k$ for $S':=S \setminus \{j\}$.  We shall construct the sets $(\mathcal{A}^S_i)_{i=0}^k$ in three stages depending on the value of $i$, verifying \ref{ax.2} and \ref{ax.3} as we go, and \ref{ax.1} at the end.
\begin{enumerate}
\item ($i<j$) Put
\begin{equation*}
\mathcal{A}^S_i:=\mathcal{A}^{S'}_i \text{ whenever } i < j
\end{equation*}
so that \ref{ax.2} and \ref{ax.3} are satisfied whenever $i<j$ since $\nu(S)_i=\nu(S')_i$ in that range.
\item ($i=j$) Put
\begin{equation*}
\mathcal{A}^S_j:=\{(a_1,\dots,a_{j-1},a) \in  \mathcal{A}^S_{j-1}\times A : x+\sum_{s \in S'}{a_s} +a  \in \overbrace{X+\dots +X}^{|S| \text{ times}}+A\}
\end{equation*}
so that \ref{ax.2} holds.  It remains to verify \ref{ax.3}.  If $(a_1,\dots,a_{j-1}) \in \mathcal{A}^S_{j-1}$ then
\begin{equation*}
\E_{a_{j},\dots,a_{k}}{1_{\mathcal{A}^{S'}_k}(a_1,\dots,a_k)} \geq \nu_k(S')\dots \nu_j(S')1_{\mathcal{A}^{S'}_{j-1}}(a_1,\dots,a_{j-1}) = 1
\end{equation*}
by induction and since $\nu_i(S')=1$ for all $j \leq i \leq k$ as $j$ was the largest element of $S$.  It follows that there is some $a_{j},\dots,a_k$s such that $(a_1,\dots,a_k) \in \mathcal{A}^{S'}_k\subset \mathcal{A}^{S'}$.  From the definition of $\mathcal{A}^{S'}$ this means that
\begin{equation*}
x+\sum_{s \in S'}{a_s} \in \overbrace{X+\dots +X}^{|S'| \text{ times}}+A,
\end{equation*}
which we note does not depend on the particular choice of $a_{j},\dots,a_k$ since the largest element of $S'$ is less than $j$.  We conclude that there is some $U(a_1,\dots,a_{j-1}) \in X+\dots +X$ (where the sum is $|S'|$-fold) such that
\begin{equation*}
-U(a_1,\dots,a_{j-1})+x+\sum_{s \in S'}{a_s} \in A.
\end{equation*}
Since $A$ is $(1-\delta)$-covered by $X$ there are at least $(1-\delta)|A|$ elements $a \in A$ such that
\begin{equation*}
-U(a_1,\dots,a_{j-1})+x+\sum_{s \in S'}{a_s} +a  \in X+A,
\end{equation*}
and so at least $(1-\delta)|A|$ elements $a \in A$ such that
\begin{equation*}
x+\sum_{s \in S'}{a_s} +a  \in \overbrace{X+\dots +X}^{|S| \text{ times}}+A;
\end{equation*}
written another way
\begin{equation*}
\E_{a_{j}}{1_{\mathcal{A}^S_{j}}(a_1,\dots,a_{j})} \geq (1-\delta)=\nu(S)_j
\end{equation*}
and so \ref{ax.3} is satisfied.
\item ($i>j$) Put
\begin{equation*}
\mathcal{A}^S_i:=\mathcal{A}^S_j\times A^{k-j}
\end{equation*}
so that \ref{ax.2} is satisfied.  It remains to note that
\begin{equation*}
\E_{a_{i}}{1_{\mathcal{A}^S_{i}}(a_1,\dots,a_{i})} = \E_{a_{i}}{1_{\mathcal{A}^S_{i-1}\times A}(a_1,\dots,a_{i})} = 1_{\mathcal{A}^S_{i-1}}(a_1,\dots,a_{i-1}),
\end{equation*}
and \ref{ax.3} is satisfied.
\end{enumerate}
Finally, note that $\mathcal{A}^S_0=\{()\}$ and
\begin{eqnarray*}
\mathcal{A}^S & = & \{a \in A^k: x+\sum_{s \in S}{a_s} \in \overbrace{X+\dots +X}^{|S| \text{ times}}+A\} \\ &= & \{a \in A^j: x+\sum_{s \in S}{a_s} \in \overbrace{X+\dots +X}^{|S| \text{ times}}+A\}\times A^{k-j}\\ & \supset & \{(a_1,\dots,a_{j-1},a) \in \mathcal{A}^S_{j-1} \times A: x+\sum_{s \in S}{a_s} \in \overbrace{X+\dots +X}^{|S| \text{ times}}+A\}\times A^{k-j} \\ & = & \mathcal{A}^S_j \times A^{k-j} = \mathcal{A}^S_k.
\end{eqnarray*}
\ref{ax.1} is proved and we have the result by induction on the largest element of $S$.
\end{proof}
The final result of this section packages up the previous lemma in a corollary that captures the aspects of (\ref{eqn.apxcov}) that we should like.  It may be worth saying that it is useful for similar reason to \cite[Proposition C.2]{tao::9}, although the link is a little obscure.

Before we state the corollary we record one final piece of notation, the precise reason for which will become clear in \S\ref{sec.chang}.  Given a finite subset $S$ of $G$ we write $\mu_S$ for the uniform probability measure supported on $S$, and for each $a \in G^r$ we write
\begin{equation*}
\mu_a:=2^{-r}\Asterisk_{i=1}^r{(\mu_{\{0_G\}} + \mu_{\{a_i\}})}.
\end{equation*}
This behaves like an approximation to the uniform measure on the group generated by $a_1,\dots,a_r$.  Indeed, if $G$ has exponent $2$ then $\mu_a=\mu_{\langle a_1,\dots,a_r\rangle}$.
\begin{corollary}\label{cor.tes}
Suppose that $\delta,\eta \in [0,1/2)$, that $A$ is $(1-\delta)$-covered by $X \ni 0_G$, $k \in \N$, and that $\mathcal{A} \subset A^k$ contains a $(1-\eta)$-large generalised sub-product.  Then
\begin{equation*}
\sum_{a \in \mathcal{A}}{\|1_{A} \ast \mu_a\|_{\ell_2(G)}^2} \geq \frac{(1-\eta)^{2k}(1-\delta)^{2k}}{|kX|}|A|^{k+1}
\end{equation*}
\end{corollary}
\begin{proof}
We apply Lemma \ref{lem.heredwant} to get that for any $x \in A$ and $S \subset [k]$ the set $\mathcal{A}^S$ contains a $(1-\delta)$-large generalised product.  By Lemma \ref{lem.heredint} it follows that $\mathcal{A}^S \cap \mathcal{A}$ contains an $\nu$-large generalised product where $\nu(S)_i=1-\eta-\delta1_S(i)$. Lemma \ref{lem.basicprop} then tells us that
\begin{eqnarray*}
\sum_{a \in \mathcal{A}}{1_{kX+A}(x+\sum_{s \in S}{a_S})} & = & |\{a \in \mathcal{A}: x+\sum_{s \in S}{a_S} \in kX+A\}|\\ & = & |\mathcal{A}^S\cap \mathcal{A}| \\ & \geq & \prod_{i=1}^k{(1-\eta-\delta1_{S}(i))}|A|^k\\ & \geq & (1-\eta)^k(1-2\delta)^{|S|}|A|^k,
\end{eqnarray*}
since $\eta,\delta<1/2$.  Note that
\begin{equation*}
1_{kX+A}(x+\sum_{s \in S}{a_s}) = \langle \mu_{\{x\}} \ast \Asterisk_{s \in S}{\mu_{\{a_s\}}}, 1_{kx+A}\rangle_{\ell_2(G)}
\end{equation*}
Averaging over $S\subset [k]$ we have that
\begin{equation*}
\frac{1}{2^k}\sum_{S \subset [k]}{1_{kX+A}(x+\sum_{s \in S}{a_s})} = \langle \mu_{\{x\}} \ast \mu_a, 1_{kx+A}\rangle_{\ell_2(G)},
\end{equation*}
and hence
\begin{eqnarray*}
\sum_{a \in \mathcal{A}}{\langle \mu_{\{x\}} \ast \mu_a, 1_{kx+A}\rangle_{\ell_2(G)}} & \geq & \frac{1}{2^k}\sum_{S \subset [k]}{ (1-\eta)^k(1-2\delta)^{|S|}|A|^k}\\ & = & \frac{1}{2^k}(1-\eta)^k|A|^k\sum_{S \subset [k]}{(1-2\delta)^{|S|}}\\ &= & \frac{1}{2^k}(1-\eta)^k|A|^k(1+(1-2\delta))^k = (1-\eta)^k(1-\delta)^k|A|^k.
\end{eqnarray*}
Summing over $x \in A$ then gives
\begin{equation*}
\sum_{a \in \mathcal{A}}{\langle 1_A\ast \mu_a,1_{kX+A}\rangle_{\ell_2(G)}}\geq (1-\eta)^k(1-\delta)^k|A|^{k+1}.
\end{equation*}
Finally, apply the Cauchy-Schwarz inequality to the inner product, and then to the outer sum to get that
\begin{eqnarray*}
(1-\eta)^k(1-\delta)^k|A|^{k+1} & \leq &\sum_{a \in \mathcal{A}}{\langle 1_A\ast \mu_a,1_{kX+A}\rangle_{\ell_2(G)}}\\ & \leq & \sum_{a \in \mathcal{A}}{\|1_A\ast \mu_a\|_{\ell_2(G)}\|1_{kX+A}\|_{\ell_2(G)}}\\ & = & |kX+A|^{1/2}\sum_{a \in \mathcal{A}}{\|1_A\ast \mu_a\|_{\ell_2(G)}}\\ & \leq & |kX+A|^{1/2}|\mathcal{A}|^{1/2}\left(\sum_{a \in \mathcal{A}}{\|1_A \ast \mu_a\|_{\ell_2(G)}^2}\right)^{1/2}\\ & \leq & |kX|^{1/2}|A|^{(k+1)/2}\left(\sum_{a \in \mathcal{A}}{\|1_A \ast \mu_a\|_{\ell_2(G)}^2}\right)^{1/2}.
\end{eqnarray*}
The result follows on rearrangement.
\end{proof}

\section{Chang's lemma}\label{sec.chang}

In this section we capture an idea of Chang from \cite[\S2]{cha::0} (recorded as Chang's covering lemma in \cite[Lemma 5.31]{taovu::}), although the connection may not be immediately obvious.  The lemma will be used in the case $h=1_A$ and combined with Corollary \ref{cor.tes} from the previous section.
\begin{lemma}\label{lem.chang}
Suppose that $h \in \ell_2(G)$ and, $\kappa\in (0,1]$, $\eta \in [0,1)$ and $k \in \N$ are parameters.  Then either
\begin{enumerate}
\item there is some $0 \leq l <k$ and some $a \in A^l$ such that the  set of $x \in A$ having
\begin{equation*}
\|h \ast \mu_{a}- \tau_x(h \ast \mu_{a})\|_{\ell_2(G)}^2 < \kappa \|h \ast \mu_{a}\|_{\ell_2(G)}^2
\end{equation*}
has size at least $\eta |A|$;
\item or there is a $(1-\eta)$-large generalised sub-product of $A^k$, $\mathcal{A}$, such that
\begin{equation*}
\|h\ast \mu_{a}\|_{\ell_2(G)}^2 \leq (1-\kappa/4)^k\|h\|_{\ell_2(G)}^2
\end{equation*}
for all $a \in \mathcal{A}$.
\end{enumerate}
\end{lemma}
\begin{proof}
Put $\mathcal{A}_0:=\{()\}$ and
\begin{equation*}
\mathcal{A}_i:=\{(a_1,\dots,a_{i}) \in \mathcal{A}_{i-1} \times A: \|h \ast \mu_{a_1,\dots,a_i}\|_{\ell_2(G)}^2 \leq (1-\kappa/4)\|h\ast \mu_{a_1,\dots,a_{i-1}}\|_{\ell_2(G)}^2\}
\end{equation*}
for $1 \leq i \leq k$ so that \ref{ax.2} holds.  Now, suppose $1 \leq i \leq k$ and $a \in \mathcal{A}_{i-1}$.  If the set of $x \in A$ such that
\begin{equation*}
\|h \ast \mu_{a} - \tau_x(h \ast \mu_{a})\|_{\ell_2(G)}^2 < \kappa \|h \ast \mu_{a}\|_{\ell_2(G)}^2
\end{equation*}
has size at least $\eta |A|$ then we terminate in the first case of the lemma with $l=i-1$.  Thus we may assume that there are at least $(1-\eta)|A|$ elements $x \in A$ such that
\begin{equation*}
\|h \ast \mu_{a} - \tau_x(h \ast \mu_{a})\|_{\ell_2(G)}^2 \geq  \kappa \|h \ast \mu_{a}\|_{\ell_2(G)}^2.
\end{equation*}
But then
\begin{eqnarray*}
4\|h \ast \mu_{(a_1,\dots,a_{i-1},x)}\|_{\ell_2(G)}^2 & =& \|h \ast \mu_{a} +\tau_x(h \ast \mu_{a})\|_{\ell_2(G)}^2\\ & = & 4\|h \ast \mu_{a}\|_{\ell_2(G)}^2 - \|h \ast \mu_{a}- \tau_x(h \ast \mu_{a})\|_{\ell_2(G)}^2\\ & \leq & (4-\kappa)\|h \ast \mu_{a}\|_{\ell_2(G)}^2,
\end{eqnarray*}
and so $(a_1,\dots,a_{i-1},x) \in \mathcal{A}_i$.  It follows that
\begin{equation*}
\E_{a_{i}}{1_{\mathcal{A}_{i}}(a_1,\dots,a_{i})} \geq (1-\eta)=(1-\eta)1_{\mathcal{A}_{i-1}}(a_1,\dots,a_{i-1})
\end{equation*}
and we have \ref{ax.3}.  Setting $\mathcal{A}=\mathcal{A}_k$ we have \ref{ax.1} for $\mathcal{A}$, and so $\mathcal{A}$ is a $(1-\eta)$-large generalised sub-product  Furthermore, if $a \in \mathcal{A}$ then
\begin{equation*}
\|h \ast \mu_{a_1,\dots,a_k}\|_{\ell_2(G)}^2 \leq (1-\kappa/4)\|h \ast \mu_{a_1,\dots,a_{k-1}}\|_{\ell_2(G)}^2 \leq \dots \leq (1-\kappa/4)^k\|h \|_{\ell_2(G)}^2
\end{equation*}
by induction and construction of the sets $\mathcal{A}_i$.  We are then in the second case of the lemma and the result is proved.
\end{proof}
It may be worth saying that it is because this lemma outputs a generalised sub-product that we had to extend Corollary \ref{cor.tes} to cover generalised sub-products; in other words, this lemma is the reason for the presence of generalised sub-products in this note.

\section{Fourier analysis and almost-invariant functions}

Fourier analysis is inextricably linked with Fre{\u\i}man's theorem and while we have not needed it so far, the introduction of convolution earlier was a clear foreshadowing of things to come.  We take a moment to record some basic definitions, but the reader may wish to refer to \cite[\S4]{taovu::} or \cite{rud::1} for a more extensive discussion.

We shall regard $G$ as a discrete group and write $\widehat{G}$ for the compact Abelian group of characters on $G$.  Given $f \in \ell_1(G)$, the \emph{Fourier transform} of $f$ is defined to be the function
\begin{equation*}
\widehat{f}:\widehat{G} \rightarrow \C; \gamma\mapsto \sum_{x \in G}{f(x)\overline{\gamma(x)}}.
\end{equation*}
The group $\wh{G}$ is naturally a compact group endowed with Haar probability measure which we shall denote $d\gamma$.  While it may seem like there is some analysis here, we are only interested in finite subsets of groups with finite exponent and so we can freely take $G$ to be finite and ignore any of this.

Following Green and Ruzsa \cite{greruz::0} we shall analyse the subgroup $\langle A\rangle$ in our problem by considering the annihilator of the large spectrum of $A$.  To make sense of this we need a couple of definitions: given a set of characters $\Gamma$, we define the \textbf{annihilator} of $\Gamma$ to be
\begin{equation*}
\Gamma^{\perp}:=\{x \in G: \gamma(x)=1 \text{ for all } \gamma \in \Gamma\}.
\end{equation*}
Given $f \in \ell_1(G)$ and $\epsilon \in (0,1]$ we define the \textbf{$\epsilon$-large spectrum} of $f$ to be
\begin{equation*}
\Spec_\epsilon(f):=\{\gamma \in \wh{G}: |\wh{f}(\gamma)| \geq \epsilon \|f\|_{\ell_1(G)}\}.
\end{equation*}
The next lemma gives us a way to contain our set in the annihilator of a suitable large spectrum, and it is here that we make essential use of the fact that $G$ has bounded exponent.  The first part of the proof is basically an argument of Green and Konyagin \cite[Lemma 3.6]{grekon::}.
\begin{lemma}\label{lem.annihilate}
Suppose that $G$ is an Abelian group of exponent $r$, $g\in \ell_1(G)$ is not identically $0$, and $\epsilon \in (0,1]$ is a parameter such that
\begin{equation*}
\|g - \tau_a(g)\|_{\ell_1(G)} \leq \epsilon\|g\|_{\ell_1(G)} \text{ for all } a \in A.
\end{equation*}
Then $A \subset \Spec_{r\epsilon}(g)^\perp$.
\end{lemma}
\begin{proof}
Suppose that $\gamma\in \Spec_{r\epsilon}(g)$ and $a \in A$. Then
\begin{eqnarray*}
r\epsilon \|g\|_{\ell_1(G)}|1-\gamma(a)| & \leq & |1-\gamma(a)||\wh{g}(\gamma)| \\ &= & |\wh{g}(\gamma) - 
\gamma(a)\wh{g}(\gamma)|\\ & = & |(g-\tau_a(g))^\wedge(\gamma)|\\ & \leq & \|g - \tau_a(g)\|_{\ell_1(G)} \leq \epsilon \|g\|_{\ell_1(G)},
\end{eqnarray*}
by the Hausdorff-Young inequality.  Dividing by $\epsilon\|g\|_{\ell_1(G)}$ (possible since $g \not \equiv 0$ and $\epsilon>0$) and rearranging we get that $|1-\gamma(a)| \leq 1/r$.  Of course since $G$ is a group of exponent $r$ it follows that $\gamma(a)$ is an $r$th root of unity and hence if it is not equal to $1$ then
\begin{equation*}
|1-\gamma(a)| \geq |1-\exp(2\pi i/r)| \geq |\sin (2\pi/r)| \geq \frac{2}{\pi}\cdot \frac{2\pi}{r} = \frac{4}{r}.
\end{equation*}
It follows that $\gamma(a)=1$ and the result is proved.
\end{proof}
With the above lemma in hand we need a supply of suitable functions $g$.  The hypothesis on $g$ look somewhat like those in the first case of Lemma \ref{lem.chang} so it should not be too surprising that we shall be combining the work of \S\ref{sec.statcov} and \S\ref{sec.chang} to act as such a supply.  The next proposition does just this and is the driving result of the whole note.
\begin{proposition}\label{prop.int}
Suppose that $G$ is an Abelian group of exponent $r$, $A \subset G$ has $|A+A| \leq K|A|$, and $\epsilon \in (0,1]$ is a parameter.  Then there is a subgroup $V$ generated by at most $O(K\epsilon^{-2}\min\{\log r,\log 2\epsilon^{-1}\})$ elements, and a non-negative function $f$ supported on $A+V$ such that
\begin{equation*}
\|f - \tau_x(f)\|_{\ell_1(G)} \leq \epsilon \|f\|_{\ell_1(G)}
\end{equation*}
for at least $\Omega(\epsilon |A|)$ elements $x \in A$.
\end{proposition}
\begin{proof}
Let $\delta$ be a parameter to be chosen later (it will just be a constant multiple of $\epsilon$).  Apply the statistical covering lemma (Lemma \ref{lem.statcov}) to the set $A$ with parameter $\delta$ to get a set $Y$ of size at most $\delta^{-1}K$ such that $A$ is $(1-\delta)$-covered by $Y$.  Let $X:=Y \cup \{0_G\}$ so that $A$ is $(1-\delta)$-covered by $X$ and $|X|\leq \delta^{-1}K+1$.

Let $k\geq \delta^{-1}K$ be a natural number to be optimised later and apply Lemma \ref{lem.chang} with $h=1_A$, $\kappa=\epsilon/4$ and $\delta$ to get that either there is some $0 \leq l < k$ and some $a \in A^l$ such that the set of $x\in A$ having
\begin{equation*}
\|1_A\ast \mu_{a} - \tau_a(1_A \ast \mu_{a})\|_{\ell_2(G)}^2 \leq \frac{\epsilon}{4}\|1_A\ast \mu_{a} \|_{\ell_2(G)}^2
\end{equation*}
has size at least $\delta |A|$, or else there is a $(1-\delta)$-large generalised sub-product $\mathcal{A}$ of $A^k$ such that
\begin{equation}\label{eqn.upper}
\|1_A\ast \mu_{a}\|_{\ell_2(G)}^2 \leq (1-\epsilon/16)^k|A| \text{ for all } a \in \mathcal{A}.
\end{equation}
Now apply Corollary \ref{cor.tes} with the set $A$ ($(1-\delta)$-covered by $X$) to get that
\begin{equation}\label{eqn.lower}
\sum_{a \in \mathcal{A}}{\|1_{A} \ast \mu_a\|_{\ell_2(G)}^2} \geq |kX|^{-1}(1-\delta)^{4k}|A|^{k+1}.
\end{equation}
Since $G$ is an Abelian group of exponent $r$ we have that $|kX| \leq r^{|X|}$.  On the other hand if $k$ is small compared with $r$ then we have a better upper bound, from the fact that $G$ is commutative, namely
\begin{eqnarray*}
|kX| \leq \binom{k+|X|-1}{|X|-1} & \leq & \exp(O(|X|(1+\log ((k+|X|)/|X|))))\\ & \leq & \exp(O((K/\delta) \log 2(k\delta/K))),
\end{eqnarray*}
where we have used the fact that $k \geq \delta^{-1}K \geq |X|-1$ in the second inequality.  It follows that
\begin{equation*}
|kX|^{1/k} \leq \exp{(O(K/\delta k)\min \{\log r,\log 2(k\delta/K)\})}.
\end{equation*}
Combining (\ref{eqn.upper}) and (\ref{eqn.lower}), dividing by $|A|^{k+1}$, and taking $k$-th roots we conclude that
\begin{equation*}
\exp{(O(K/\delta k)\min \{\log r,\log 2(k\delta/K)\})}(1-\epsilon/16)\geq (1-\delta)^4.
\end{equation*}
We can choose $k = O(K\delta^{-2}\min\{\log r,\log 2\delta^{-1}\})$ such that the first term on the left is at most $1+\delta$, and it follows that we can then take $\delta = \Omega(\epsilon)$ to get a contradiction.  This contradiction means that we must have been in the first case of Lemma \ref{lem.chang} at some point \emph{i.e.} there is some
\begin{equation*}
l<k = O(K\delta^{-2}\min\{\log r,\log 2\delta^{-1}\})=O(K\epsilon^{-2}\min\{\log r,\log 2\epsilon^{-1}\})
\end{equation*}
and some $a \in A^l$ such that
\begin{equation*}
\|1_A\ast \mu_{a} - \tau_x(1_A \ast \mu_{a})\|_{\ell_2(G)}^2 \leq \frac{\epsilon}{4}\|1_A\ast \mu_{a'} \|_{\ell_2(G)}^2
\end{equation*}
for $\Omega(\epsilon |A|)$ elements $x \in A$.  We put $V:=\langle a_1,\dots,a_l\rangle$, and see that $V$ is generated by the claimed number of elements, and $f:=(1_A \ast \mu_{a})^2$ so that $f$ is supported on $A+V$.  It remains to note that by the triangle inequality and the fact that $\tau_x$ is an isometry we have
\begin{eqnarray*}
\|f - \tau_x(f)\|_{\ell_1(G)} & \leq&  |\langle 1_A \ast \mu_{a}, 1_A \ast \mu_{a}-\tau_x(1_A \ast \mu_{a})\rangle|\\ & & + |\langle \tau_x(1_A \ast \mu_{a}), 1_A \ast \mu_{a}-\tau_x(1_A \ast \mu_{a})\rangle| \\ & = & 2\|1_A \ast \mu_{a}-\tau_x(1_A \ast \mu_{a})\|_{\ell_2(G)}^2\\ & & + 2\|1_A \ast \mu_{a}-\tau_{-x}(1_A \ast \mu_{a})\|_{\ell_2(G)}^2\\ & \leq & 4\cdot \frac{\epsilon}{4} \|1_A \ast \mu_{a}\|_{\ell_2(G)}^2 = \epsilon \|f\|_{\ell_1(G)}.
\end{eqnarray*}
The result is proved.
\end{proof}
Although the proposition makes use of the fact that $G$ has bounded exponent this is not really essential and it can be recast as a useful statement in more general settings too.

Finally, while the proposition does provide functions satisfying the hypothesis of Lemma \ref{lem.annihilate}, they are are only useful if we can also show that the annihilator of the large spectrum of these functions is small; the next lemma does this.  Its basis is an idea introduced to Fre{\u\i}man-type problems by Green and Ruzsa in \cite{greruz::0} in a way closely related to work of Schoen \cite{sch::0}.
\begin{lemma}\label{lem.big}
Suppose that $G$ is an Abelian group, $\emptyset \neq A \subset G$ has $|A+A| \leq K|A|$, $\epsilon \in (0,1/2]$ is a parameter, and $0 \not\equiv h\in \ell_1(G)$ is a non-negative function supported on $A$ such that
\begin{equation*}
\|h - \tau_a(h)\|_{\ell_1(G)} \leq \epsilon\|h\|_{\ell_1(G)}
\end{equation*}
for all $a \in A'$.  Then for any non-negative $0 \not\equiv g \in \ell_1(G)$ supported on $A'$ we have
\begin{equation*}
|\Spec_{1/4K^{2\epsilon}}(g)^\perp| \leq 4K|A|.
\end{equation*}
\end{lemma}
\begin{proof}
Let $k:=\lfloor \epsilon^{-1}/2\rfloor$ and note that by the triangle inequality and the fact that $\tau_x$ is an isometry we have
\begin{equation*}
\|h - \tau_{-x}(h)\|_{\ell_1(G)}=\|h - \tau_x(h)\|_{\ell_1(G)} \leq \frac{1}{2}\|h\|_{\ell_1(G)} \text{ for all } x \in kA'.
\end{equation*}
Since $h$ is non-negative and the support of $h$ is contained in $A$ we have that
\begin{eqnarray*}
\|h\|_{\ell_1(G)}|A| &=& \langle h, 1_{A+A}\ast 1_{-A}\rangle_{\ell_2(G)}\\ & = & \langle \tau_{-x}(h),1_{A+A}\ast 1_{-A}\rangle_{\ell_2(G)} + \langle (h-\tau_{x-}(h)),1_{A+A}\ast 1_{-A}\rangle_{\ell_2(G)}\\ & \leq &  \langle \tau_{-x}(h),1_{A+A}\ast 1_{-A}\rangle_{\ell_2(G)} + \frac{1}{2}\|h\|_{\ell_1(G)}|A|
\end{eqnarray*}
for any $x \in kA'$.  Summing against $\overbrace{g \ast \dots \ast g}^{k \text{ times}}(x)$ (which has support on $kA'$ and is non-negative), we get that
\begin{equation}\label{eqn.base}
\langle h \ast \overbrace{g \ast \dots \ast g}^{k \text{ times}},1_{A+A}\ast 1_{-A}\rangle_{\ell_2(G)} \geq  \frac{1}{2}\|h\|_{\ell_1(G)}|A|\|g\|_{\ell_1(G)}^k.
\end{equation}
We can then apply Plancherel's theorem to see that
\begin{equation*}
\int{\wh{h}(\gamma)\wh{g}(\gamma)^k\overline{\wh{1_{A+A}}(\gamma)}\wh{1_A}(\gamma)d\gamma} \geq \frac{1}{2}\|h\|_{\ell_1(G)}|A|\|g\|_{\ell_1(G)}^k.
\end{equation*}
Write $S:=\Spec_{1/4K^{2\epsilon}}(g)$ and, with explanation of the passage between the lines in the following paragraph, we then have
\begin{eqnarray}
\nonumber \int_{\wh{G}\setminus S}{|\wh{h}(\gamma)||\wh{g}(\gamma)|^k|\wh{1_{A+A}}(\gamma)||\wh{1_A}(\gamma)|d\gamma} & \leq &\left(\frac{\|g\|_{\ell_1(G)}}{4K^{2\epsilon}}\right)^k \int{|\wh{h}(\gamma)||\wh{1_{A+A}}(\gamma)||\wh{1_A}(\gamma)|d\gamma}\\ \nonumber & \leq &  \frac{\|g\|_{\ell_1(G)}^k}{4^kK^{2\epsilon k}}\|h\|_{\ell_1(G)}\int{|\wh{1_{A+A}}(\gamma)||\wh{1_A}(\gamma)|d\gamma}\\ \nonumber & \leq & \frac{\|g\|_{\ell_1(G)}^k}{4^kK^{2\epsilon k}}\|h\|_{\ell_1(G)}\\ \nonumber & & \times\left(\int{|\wh{1_{A+A}}(\gamma)|^2d\gamma}\right)^{1/2}\left(\int{|\wh{1_A}(\gamma)|^2d\gamma}\right)^{1/2}\\ \nonumber & = & \frac{\|g\|_{\ell_1(G)}^k}{4^kK^{2\epsilon k}}\|h\|_{\ell_1(G)}\sqrt{|A+A||A|}\\ \nonumber  & \leq & \frac{1}{4}\|g\|_{\ell_1(G)}^k\|h\|_{\ell_1(G)}|A|K^{1/2-2k\epsilon}\\ \label{eqn.ots} & \leq & \frac{1}{4}\|h\|_{\ell_1(G)}|A|\|g\|_{\ell_1(G)}^k.
\end{eqnarray}
The first inequality is the definition of $S$; the second inequality is the Hausdorff-Young inequality applied to $h$; the third inequality is the Cauchy-Schwarz inequality; the following equality is Parseval's theorem; and we then finish the chain by noting that $4^{-k} \leq 4^{-1}$, $|A+A| \leq K|A|$, and $1/2 - 2k\epsilon \leq 0$.

It remains to note, again with explanations afterwards, that
\begin{eqnarray*}
\|h\|_{\ell_1(G)}\|g\|_{\ell_1(G)}^k\int_{S}{|\wh{1_{A+A}}(\gamma)||\wh{1_A}(\gamma)|d\gamma} &\geq & \int_{S}{|\wh{h}(\gamma)||\wh{g}(\gamma)|^k|\overline{\wh{1_{A+A}}(\gamma)}||\wh{1_A}(\gamma)|d\gamma}\\  & \geq &
\left|\int_{S}{\wh{h}(\gamma)\wh{g}(\gamma)^k\overline{\wh{1_{A+A}}(\gamma)}\wh{1_A}(\gamma)d\gamma}\right|\\ & \geq & \left|\int{\wh{h}(\gamma)\wh{g}(\gamma)^k\overline{\wh{1_{A+A}}(\gamma)}\wh{1_A}(\gamma)d\gamma}\right|\\ & &  - \int_{\wh{G}\setminus S}{|\wh{h}(\gamma)||\wh{g}(\gamma)|^k|\wh{1_{A+A}}(\gamma)||\wh{1_A}(\gamma)|d\gamma}\\ & \geq  &\frac{1}{2}\|h\|_{\ell_1(G)}|A|\|g\|_{\ell_1(G)}^k - \frac{1}{4}\|h\|_{\ell_1(G)}|A|\|g\|_{\ell_1(G)}^k\\ & = & \frac{1}{4}\|h\|_{\ell_1(G)}|A|\|g\|_{\ell_1(G)}^k.
\end{eqnarray*}
The first inequality is the Hausdorff-Young inequality in $h$ and $g$; the second is the integral triangle inequality; the third is the triangle inequality; and the final inequality then inserts (\ref{eqn.base}) and (\ref{eqn.ots}).

Dividing out by $\|h\|_{\ell_1(G)}$ and $\|g\|_{\ell_1(G)}^k$ (both of which are non-zero since $g$ and $h$ are non-trivial) we see that
\begin{equation*}
\int_{S}{|\wh{1_{A+A}}(\gamma)||\wh{1_A}(\gamma)|d\gamma} \geq \frac{1}{4}|A|.
\end{equation*}
Let $V$ be any finite subgroup of $\Spec_{1/4K^{2\epsilon}}(g)^\perp$ and note that $\wh{1_V}(\gamma)=|V|$ for all $\gamma \in S$.  It follows that
\begin{eqnarray*}
\frac{1}{4}|A||V|^2 &\leq &\int_S{|\wh{1_{A+A}}(\gamma)||\wh{1_A}(\gamma)||\wh{1_V}(\gamma)|^2d\gamma}\\ & \leq & \int{|\wh{1_{A+A}}(\gamma)||\wh{1_A}(\gamma)||\wh{1_V}(\gamma)|^2d\gamma} \leq |A||A+A||V|
\end{eqnarray*}
by the Hausdorff-Young inequality in $1_{A+A}$ and $1_A$, and Parseval's theorem in $1_V$. We conclude that $|V|\leq 4K|A|$.  It follows that $\Spec_{1/4K^{\epsilon}}(g)^\perp$ is finite and hence satisfies the required bound.
\end{proof}

\section{Proof of the main theorem}

Before proving our main result we need to record one more ingredient.  In \cite{pet::} Petridis found a fantastic new proof of Pl{\" u}nnecke's inequality \cite{plu::} (see also \cite{ruz::5}) which proceeded via the following lemma.
\begin{lemma}[Petridis' lemma {\cite[Proposition 2.1]{pet::}}]  Suppose that $A,B \subset G$ are finite sets with $|A+B| \leq K|B|$, and $Z \subset B$ is non-empty with $|A+Z|/|Z|$ minimal.  Then
\begin{equation*}
|A+Z+C| \leq K|Z+C| \text{ for all finite }C \subset G.
\end{equation*}
\end{lemma}
We shall not discuss the proof of this here, although it inspired the proof of Lemma \ref{lem.statcov}.  Indeed, the genesis of this note centred around trying to use Petridis' arguments to give a proof of Ruzsa's conjecture.  That approach failed, at least in part because Petridis' arguments actually work just as well for non-Abelian groups as they do for Abelian groups, and Ruzsa's conjecture is essentially Abelian.  

Finally, then, we turn to our proof.
\begin{proof}[Proof of Theorem \ref{thm.main}]
Let $Z \subset A$ be such that $|A+Z|/|Z|$ is minimal.  In particular, $|A+Z| \leq K|Z|$ and $|Z+Z| \leq K|Z|$.  We apply Proposition \ref{prop.int} to $Z$ with a parameter $\epsilon$ (to be optimised later, ending up being $\Omega(1/\log K)$) to get a subgroup $V$ generated by at most $O(K\epsilon^{-2}\min\{\log r,\log 2\epsilon^{-1}\})$ elements and a non-negative function $f\not \equiv 0$ with support on $Z+V$ such that
\begin{equation*}
\|f - \tau_z(f)\|_{\ell_1(G)} \leq \epsilon \|f\|_{\ell_1(G)}
\end{equation*}
for at least $\Omega(\epsilon |Z|)$ elements $z \in Z$; call the set of such $z$s $Z'$.  Thus 
\begin{equation*}
|Z'+Z'| \leq |Z+Z| \leq K|Z| =O(\epsilon^{-1}K|Z'|).
\end{equation*}
Now apply Proposition \ref{prop.int} again, but this time to the set $Z'$ with a parameter $\eta$ (again, to be optimised later, but this time it will end up being $\Omega(1/r)$).  This gives us a subgroup $V'$ generated by at most $O(K\epsilon^{-1}\eta^{-2}\min\{\log r,\log 2\eta^{-1}\})$ elements and a non-negative function $g \not \equiv 0$ with support on $Z'+V'$ such that
\begin{equation*}
\|g - \tau_z(g)\|_{\ell_1(G)} \leq \eta \|g\|_{\ell_1(G)}
\end{equation*}
for at least $\Omega(\eta |Z'|)$ elements $z \in Z'$; call the set of such $z$s $Z''$. We shall return to $Z''$ later, but now we turn to showing that $g$ has large Fourier coefficients.

Let $h:=f \ast \mu_{V'}$ \emph{i.e.}
\begin{equation*}
h(x):=\int{f(x-y)d\mu_{V'}(y)} = \int{\tau_{-y}(f)(x)d\mu_{V'}(y)} \text{ for all } x \in G.
\end{equation*}
It follows immediately that $h$ is invariant under translation by elements of $V'$.  Suppose that $z \in Z'+V'$, so that $z=z'+v'$ where $z'\in Z'$ and $v'\in V'$.  Then using the fact that $h$ is invariant under translation by elements of $V'$; linearity of $\tau$; the integral Minkowski inequality; and finally the isometry of $\tau$ we get that
\begin{eqnarray*}
\|h - \tau_z(h)\|_{\ell_1(G)}& =& \|h - \tau_{z'}(\tau_{v'}(h))\|_{\ell_1(G)}\\ &= & \|h - \tau_{z'}(h)\|_{\ell_1(G)}\\ & = & \left\|\int{\tau_{-y}(f)d\mu_{V'}(y)}- \tau_{z'}\left(\int{\tau_{-y}(f)d\mu_{V'}(y)}\right)\right\|_{\ell_1(G)}\\ & = & \left\|\int{\tau_{-y}(f-\tau_{z'}(f))d\mu_{V'}(y)}\right\|_{\ell_1(G)}\\ & \leq & \int{\|\tau_{-y}(f-\tau_{z'}(f))\|_{\ell_1(G)}d\mu_{V'}(y)}\\ & = & \int{\|f-\tau_{z'}f\|_{\ell_1(G)}d\mu_{V'}(y)}\leq  \epsilon \|f\|_{\ell_1(G)}.
\end{eqnarray*}
To summarise:
\begin{equation*}
\|h - \tau_z(h)\|_{\ell_1(G)} \leq \epsilon \|f\|_{\ell_1(G)} \text{ for all }z \in Z'+V'
\end{equation*}
Additionally $h$ is supported on $Z+V+V'$, and by Petridis' Lemma we have that
\begin{eqnarray*}
|(Z+V+V') + (Z+V+V')| & \leq & |(A+V+V') + (Z+V+V')|\\ & = & |A+Z+(V+V')| \leq K|Z+V+V'|.
\end{eqnarray*}
Take $\epsilon =1/4\log 2K$ and apply Lemma \ref{lem.big} to the set $Z+V+V'$ and the functions $h$ and $g$ to get that
\begin{eqnarray*}
|\Spec_{1/4\sqrt{e}}(g)^\perp|&\leq &4K|Z+V+V'|\\ & \leq & 4K|Z| |V| |V'|\\ & = & \exp(O(K(\log 2K)((\log 2K)(\log r)+\eta^{-2})))|A|.
\end{eqnarray*}
Now, if we put $\eta=1/4r\sqrt{e}$ then by Lemma \ref{lem.annihilate} we have that $Z'' \subset \Spec_{1/4r\sqrt{e}}(g)^\perp$.  On the other hand $Z'' \subset Z' \subset Z$ and so
\begin{equation*}
|A+Z''| \leq K|Z| =O( K\epsilon^{-1}\eta^{-1}|Z''|).
\end{equation*}
Let $Z''' \subset Z''$ be such that $|A+Z'''|/|Z'''|$ is minimal and put $V''':=\langle Z'''\rangle \subset \langle Z''\rangle \Spec_{1/4r\sqrt{e}}(g)^\perp$ so that
\begin{equation*}
|V'''| \leq  \exp(O(K(\log 2K)((\log 2K)\log r+r^2)))|A|,
\end{equation*}
and note that by Petridis' lemma we have
\begin{eqnarray*}
|A+V'''| = |A+Z'''+V'''| & =& O( K\epsilon^{-1}\eta^{-1}|Z'''+V'''|)\\  & = & O(K\epsilon^{-1}\eta^{-1}|V'''|)=O(Kr(\log 2K)|V'''|).
\end{eqnarray*}
It follows that $A+V'''$ contains at most $O(Kr\log 2K)$ cosets of $V'''$ and hence
\begin{equation*}
|\langle A \rangle| \leq r^{O(Kr\log 2K)}|V'''| = \exp(O(K(\log 2K)((\log 2K)(\log r) + r^2)))|A|
\end{equation*}
as required.
\end{proof}

\bibliographystyle{halpha}

\bibliography{references}

\end{document}